\documentclass[reqno,12pt]{amsart}
\newcommand{\CC}{{\mathbb C}}

\def\bege{\begin{equation}} \def\ende{\end{equation}}
\def\begr{\begin{eqnarray}} \def\endr{\end{eqnarray}}
\usepackage{amsmath}
\usepackage{amssymb}
\usepackage{cite}

\usepackage{bbm}
\usepackage{amsmath}
\usepackage{txfonts}
\usepackage{stmaryrd}
\usepackage{amssymb}
\usepackage{amsfonts}
\usepackage{times}
\usepackage{mathrsfs}

\usepackage{amsthm}
\usepackage{enumerate}

\usepackage{framed}
\usepackage{lipsum}
\usepackage{color}

\def\CC{ \mathbb{C}}

\def\T{\mathcal{T}}

\def\D{\mathbb{D}}
\def\N{\mathbb N}

\def\hD{\hat{\mathcal{D}}}
\def\dD{\mathcal{D}}
\def\vp{\varphi}
\def\om{\omega}

\def\p{{\prime}}

\def\begr{\begin{eqnarray}} \def\endr{\end{eqnarray}}

\def\ol{\overline}
\newtheorem{Lemma}{Lemma}%[section]
\newtheorem{Theorem}[Lemma]{Theorem}

\newtheorem{Remark}[Lemma]{Remark}

\newcounter{other}            % Questions get letters
\newtheorem{otherth}[other]{Theorem}              % Other papers' theorems
\newtheorem*{th2p}{Theorem \ref{th2}$^\p$}

\begin{document}
\title[]{Stevi\'c-Sharma type operators between Bergman spaces induced by doubling weights}

\author{ Juntao Du,  Songxiao Li$\dagger$  and Zuoling Liu}

\address{Juntao Du\\ Department of mathematics, Guangdong University of Petrochemical Technology, Maoming, Guangdong, 525000,  P. R. China.}
 \email{jtdu007@163.com  }

\address{Songxiao Li\\ Department of mathematics, Shantou University, Shantou, Guangdong, 515063,  P. R. China.}
\email{jyulsx@163.com}
\address{Zuoling Liu \\  Department of Mathematics, Shantou University, Shantou, Guangdong, 515063,  P. R. China }
\email{zlliustu@163.com, 16zlliu1@stu.edu.cn}

 \begin{abstract}
 Using Khinchin's inequality, Ger$\check{\mbox{s}}$gorin's theorem and the atomic decomposition of Bergman spaces,  we  estimate the norm and essential norm of Stevi\'c-Sharma type operators  from   weighted Bergman spaces $A_\omega^p$  to $A_\mu^q$ and the  sum of weighted differentiation composition operators with different symbols from   weighted Bergman spaces $A_\omega^p$  to  $H^\infty$.  The estimates of those between Bergman spaces remove all the restrictions of a  result in [Appl. Math. Comput.,  {\bf 217}(2011), 8115--8125].
As a by-product, we also get an  interpolation theorem for Bergman spaces induced by doubling weights.
 \thanks{$\dagger$ Corresponding author.}
 \thanks{The work was supported by NNSF of China (Nos. 12371131 and 12271328),  Guangdong Basic and Applied Basic Research Foundation (No. 2022A1515012117) and Projects of Talents Recruitment of GDUPT(No. 2022rcyj2008), Project of Science and Technology of Maoming (No. 2023417) and STU Scientific Research Initiation Grant (No. NTF23004).}
\end{abstract}
 \maketitle

\section{Introduction}
Let $H(\D)$ denote  the space of all analytic functions in  the open unit disc $\D=\{z\in\CC:|z|<1\}$. For a nonnegative function $\om\in L^1([0,1])$, the extension to $\D$, defined by $\om(z)=\om(|z|)$ for all $z\in\D$, is called a radial weight. The set of doubling weights, denoted by $\hat{\mathcal{D}}$,  consists of all radial weights $\omega$ such that (see \cite{Pj2015})
% the tail integral $\hat{\omega}$ satisfies the doubling property
 $$\hat{\omega}(r)\leq C \hat{\omega}\left(\frac{r+1}{2}\right)$$
for all $0\leq r <1$ and constant $C=C(\omega)\geq1$. Here  $\hat{\om}(z)=\int_{|z|}^1\om(t)dt $.
Moreover, if $\omega\in\hat{\mathcal{D}}$ and satisfies
\begin{align}\label{0405-1}
\hat{\omega}(r)\geq C\hat{\omega}\left(1-\frac{1-r}{K}\right)
\end{align}
for  all $0\leq r <1$ and some constants $K=K(\omega)>1$ and $C=C(\omega)>1$, we refer to $\omega$ as a two-sides doubling weight and denote it by $\omega\in\dD$. For any $\lambda\in\D$,  the  Carleson square at $\lambda\in\D$ is defined by
$$S(\lambda)=\left\{re^{\mathrm{i\theta}}:|\lambda|\leq r<1, |\mbox{Arg } \lambda-\theta|<\frac{1-|\lambda|}{2}\right\}.$$
For a radial weight $\om$, let $\om(S(\lambda))=\int_{S(\lambda)}\om(z)dA(z)$.
Obviously, $\om(S(\lambda))\approx (1-|\lambda|)\hat{\om}(\lambda)$.
See  \cite{Pj2015,  PjRj2021adv, PjRjSk2021jga} and references therein for more properties of doubling weight.

For $0<p<\infty$ and a given $\omega\in\hat{\mathcal{D}}$, the Bergman space $A_\omega^p$ with doubling weight consists of all  functions $f\in H(\D)$ such that
$$\|f\|_{A_\omega^p}^p=\int_\mathbb{D} |f(z)|^p\omega(z)dA(z)<\infty,$$
where $dA$ is the normalized area measure on $\mathbb{D}$.
As usual, we denote $A_\alpha^p$ as the standard weighted Bergman space induced by the radial weight $\omega(z)=(\alpha+1)(1 - |z|^2)^\alpha$ with $-1<\alpha<\infty$.
Throughout this paper, we assume that  $\hat{\om}(z) >0$ for all $z\in\D$. Otherwise $A_\om^p=H(\D)$. Let $H^\infty$ denote the bounded analytic function space, i.e.,
$$ H^\infty=\left\{f\in H(\D):\|f\|_{H^\infty}=\sup_{z\in\D}|f(z)|<\infty\right\}.$$

Let  $S(\D)$ be the set of all analytic self-maps of $\D$. For $n\in\N\cup\{0\},\vp\in S(\D)$, and $u\in H(\D)$, the generalized weighted composition operator $uD_\vp^{(n)}$ is defined by
$$uD_\vp^{(n)} f=u \left(f^{(n)}\circ \vp\right), \quad f\in H(\D). $$
The operator $uD_\vp^{(n)}$ was introduced by Zhu in \cite{ZXL2007}.  The generalized  weighted composition operator is also called a weighted differentiation composition operator (see \cite{st1, st2, st3}).  When $n=0$, $uD_\vp^{(n)}$ is the weighted composition operator $uC_\vp$. In particular, when $n=0$ and $u\equiv 1$,  $uD_\vp^{(n)}$ is the composition operator $C_\vp$. By using the pull-back measure, the first two  authors of this paper and Shi \cite{DjLsSy2020ms}  estimated the norm and essential norm of weighted composition operators between Bergman spaces induced by doubling weights. At the same time, Liu \cite{Lb2021bams} independently characterized the boundedness and compactness of  weighted differentiation composition operator $uD_\vp^{(n)}:A_\om^p\to L_\upsilon^q$ when $0<p,q<\infty$, $\om\in\dD$ and $\upsilon$ is a positive Borel measure on $\D$. For more discussion on composition operators and  weighted composition operators, we refer to \cite{CM1995,EK2023, GMR2023,SsSaBa2011amc,SU2014, Zhu1} and the references therein.  When $u\equiv 1$, $uD_\vp^{(n)}$ is the differentiation  composition operator $D_\vp^{(n)}$. When $u\equiv 1$ and $\vp(z)=z$, $uD_\vp^{(n)}$ is the $n$-th differentiation operator $D^{(n)}$.  So,  the generalized weighted composition operator   attracted a lot of attentions since it covers a lot of classical operators. See \cite{st1, st2, st3, ZXL2007, zxl2, zxl5, ZXL2019} for further information and results on generalized weighted composition operators on analytic function spaces.

     In 2011, Stevi\'c, Sharma and Bhat \cite{SsSaBa2011amc} introduced a operator $T_{u_0,u_1,\vp}$ as follows.
 $$T_{u_0,u_1,\vp}=u_0D_\vp^{(0)}+u_1 D_\vp^{(1)}.$$
 This operator and its extension $\sum_{k=0}^n u_k D_{\vp_k}^{(n)} $ are called Stevi\'c-Sharma type operators by some authors, where $\{u_k\}_{k=0}^n \subset H(\D)$ and $\{\vp_k\}_{k=0}^n\subset S(\D)$.  In \cite{SsSaBa2011amc},  Stevi\'c, Sharma and Bhat characterized the boundedness of $T_{u_0,u_1,\vp}:A_\alpha^p\to A_\alpha^p$ with the assumption
 \begin{align}\label{0914-1}
 u_0\in H^\infty\,\,\,\mbox{or}\,\,\,\sup_{z\in\D}\frac{|u_1(z)|}{1-|\vp(z)|^2}<\infty.
 \end{align}
Two natural questions are raised.
 \begin{enumerate}
   \item[{\bf Q1.}] Whether the condition (\ref{0914-1}) can be removed?
   \item[{\bf Q2.}] What about the operator $u_0 D_{\vp_0}^{(0)}+u_1D_{\vp_1}^{(1)}$ when $\vp_0\neq \vp_1$?
 \end{enumerate}
 See, for example,   \cite{FG2022,G2022f,WWG2020,WWG2020B, YyLy2015caot,ZfLy2018caot} for some investigations about these operators.

By using Khinchin's inequality, Ger$\check{\mbox{s}}$gorin's theorem, and the atomic decomposition of weighted Bergman spaces,
we give a positive answer to  question  {\bf Q1}. Moreover, we extend it to a more general case and  completely estimate the norm and essential norm of the operator
$$T_{n,\varphi,\vec{u}} =\sum_{k=0}^n u_k D_\vp^{(k)} $$
from  $A_\om^p$ to  $A_\mu^q$ with $\om,\mu\in\dD$ and $1\leq p,q<\infty$, where $\vec{u}=(u_0, u_1,\cdot\cdot\cdot, u_n)$ with $\{u_k\}_{k=0}^n\subset H(\D)$  and $\vp\in S(\D)$. The first  result of this paper is stated as follows.

\begin{Theorem}\label{th1} Suppose $1\leq p,q<\infty,\om,\mu\in\dD$, $\vp\in S(\D)$ and  $\{u_k\}_{k=0}^n\subset H(\D)$. Then,
$$\|T_{n,\varphi,\vec{u}}\|_{A_\om^p\to A_\mu^q}\approx \sum_{k=0}^n\|u_k D_\vp^{(k)}\|_{A_\om^p\to A_\mu^q}.$$
Moreover, if $T_{n,\varphi,\vec{u}}:A_\om^p\to A_\mu^q$ is bounded, then
$$\|T_{n,\varphi,\vec{u}}\|_{e,A_\om^p\to A_\mu^q}\approx \sum_{k=0}^n\|u_kD_\vp^{(k)}\|_{e,A_\om^p\to A_\mu^q}.$$
\end{Theorem}

By Remark \ref{0916-2} and Theorem \ref{thB} in Section 2, Theorem 1 completely characterizes the norm and essential norm of $T_{n,\varphi,\vec{u}}:A_\om^p\to A_\mu^q$

Recall that the essential norm of a bounded operator $T:X\to Y$  is defined by
$$\|T\|_{e,X\to Y}=\inf\Big\{\|T-K\|_{X\to Y};K:X\to Y \mbox{ is compact}\Big\}.$$
Here $X$ and $Y$ are Banach spaces. Obviously, $T$ is compact if and only if $\|T\|_{e,X\to Y}=0$.

For the question   {\bf Q2},   Acharyya and Ferguson  \cite{AsFt2019caot} characterized the  compactness of the operator $\T_{n,\vec{\varphi},\vec{u}}:A_\alpha^p\to H^\infty$, where $$\T_{n,\vec{\varphi},\vec{u}}=\sum_{k=0}^n u_k D_{\vp_k}^{(k)} .$$ Here   $\vec{u}=(u_0, u_1,\cdot\cdot\cdot, u_n)$ with $\{u_k\}_{k=0}^n\subset H(\D)$  and $\vec{\vp}=(\vp_0, \vp_1,\cdot\cdot\cdot, \vp_n)$ with $\{\vp_k\}_{k=0}^n\subset S(\D)$.    In this paper, we  extend \cite[Theorem 2]{AsFt2019caot} to the case of Bergman spaces $A_\om^p$ with $\om\in\hD$.

\begin{Theorem}\label{th2}  Suppose $n\in\N\cup\{0\}$, $1\leq p<\infty$, $\om\in\hD$, $\{u_k\}_{k=0}^n\subset H(\D)$ and $\{\vp_k\}_{k=0}^n\subset S(\D)$.
Then,   $$\|\T_{n,\vec{\varphi},\vec{u}}\|_{A_\om^p\to H^\infty}\approx \sum\limits_{k=0}^n \sup\limits_{z\in\D} \frac{|u_k(z)|}{(1-|\vp_k(z)|^2)^k\om(S(\vp_k(z)))^\frac{1}{p}}.$$
Moreover, if $\T_{n,\vec{\varphi},\vec{u}}:{A_\om^p\to H^\infty}$  is bounded, then
  $$\|\T_{n,\vec{\varphi},\vec{u}}\|_{e,A_\om^p\to H^\infty}\approx \sum\limits_{k=0}^n \limsup\limits_{|\vp_k(z)|\to 1} \frac{|u_k(z)|}{(1-|\vp_k(z)|^2)^k\om(S(\vp_k(z)))^\frac{1}{p}}.$$
\end{Theorem}

By Theorem \ref{th2}, it is easy to check that
$$\|u_kD_{\vp_k}^{(k)}\|_{A_\om^p\to H^\infty} \approx \sup_{z\in\D} \frac{|u_k(z)|}{(1-|\vp_k(z)|^2)^k\om(S(\vp_k(z)))^\frac{1}{p}}$$
and
$$\|u_kD_{\vp_k}^{(k)}\|_{e,A_\om^p\to H^\infty} \approx \limsup_{|\vp_k(z)|\to 1} \frac{|u_k(z)|}{(1-|\vp_k(z)|^2)^k\om(S(\vp_k(z)))^\frac{1}{p}}.$$
Therefore, Theorem \ref{th2} can be stated similarly as Theorem \ref{th1}.

\begin{th2p} Suppose $n\in\N\cup\{0\}$, $1\leq p<\infty$, $\om\in\hD$, $\{u_k\}_{k=0}^n\subset H(\D)$ and $\{\vp_k\}_{k=0}^n\subset S(\D)$.
Then,
  $$\|\T_{n,\vec{\varphi},\vec{u}}\|_{A_\om^p\to H^\infty}\approx \sum\limits_{k=0}^n \|u_kD_{\vp_k}^{(k)}\|_{A_\om^p\to H^\infty}.$$
Moreover, if $\T_{n,\vec{\varphi},\vec{u}}:{A_\om^p\to H^\infty}$  is bounded, then
  $$\|\T_{n,\vec{\varphi},\vec{u}}\|_{e,A_\om^p\to H^\infty}\approx  \sum\limits_{k=0}^n \|u_kD_{\vp_k}^{(k)}\|_{e,A_\om^p\to H^\infty}.$$
\end{th2p}

The sufficiency of  Theorems \ref{th1} and \ref{th2} are easy to verify.  To establish the necessity part, we need the following interpolation theorem, which has its own interesting.
Here and henceforth,   $\delta_{ij}$  is the Dirac function, that is,
 $\delta_{ij}=1$ when $i=j$ and $\delta_{ij}=0$ when $i\neq j$.

\begin{Theorem}\label{th3}  Let $1\leq p<\infty$, $n\in\N\cup\{0\}$, $\om\in \hD$. Then there is a positive constant $C$ such that  for all $\Lambda=\{\lambda_j\}_{j=0}^n\subset\D$ and $J\in\{0,1,\cdots,n\}$, there exists $f_{\Lambda,J}\in A_\om^p$ satisfying $\|f_{\Lambda,J}\|_{A_\om^p}\leq C$   and
\begin{align}\label{IC}
f_{\Lambda,J}^{(j)}(\lambda_j)=\frac{\delta_{jJ}}{(1-|\lambda_j|^2)^j\om(S(\lambda_j))^\frac{1}{p}},\,\,j=0,1,\cdots,n.
\end{align}
Moreover, if $J$ is fixed, the functions $\{f_{\Lambda,J}\}$ converge to 0 uniformly on compact subsets of $\D$ as $|\lambda_J|\to 1$.
\end{Theorem}

The rest of this paper is organized as follows. In Section 2, we will gather the necessary preliminaries. Sections 3, 4, and 5 are dedicated to the proofs of Theorem \ref{th1}, Theorem \ref{th3}, and Theorem \ref{th2}, respectively.

Throughout this paper, the letter $C$ will represent constants, which may vary from one occurrence to another. For two positive functions $f$ and $g$, we use the notation $f \lesssim g$ to denote that there exists a positive constant $C$, independent of the arguments, such that $f \leq Cg$. Similarly, $f \approx g$ indicates that $f \lesssim g$ and $g \lesssim f$.\\

\section{preliminaries}
In this section, we state some lemmas which  will be used in the proof of main results of this paper.
For brief, for any given $\alpha>0$, let
$$\om_{[\alpha]}(z)=(1-|z|^2)^\alpha\om(z),\,\,\,\,z\in\D.$$

\begin{Lemma}\label{0405-2}  Suppose  $\alpha>0, \om\in\dD$. Then $\om_{[\alpha]}\in\dD$ and $\widehat{\om_{[\alpha]}}\approx \hat{\om}_{[\alpha]}$.
\end{Lemma}

\begin{proof} For brief, let $\eta=\om_{[\alpha]}$.
Let $C$ and $K$ be those in (\ref{0405-1}). For all $t\in[0,1)$, we have
$\hat{\eta}(t)\lesssim (1-t)^{\alpha}\hat{\om}(t)$ and
$$\hat{\eta}(t)\geq \int_t^{1-\frac{1-t}{K}}(1-s^2)^{\alpha}\om(s)ds
\approx (1-t)^{\alpha}\left(\hat{\om}(t)-\hat{\om}\bigg(1-\frac{1-t}{K}\bigg) \right)\gtrsim (1-t)^{\alpha}\hat{\om}(t).$$
Then, we have $\hat{\eta}(t)\approx (1-t)^{\alpha}\hat{\om}(t)$ and then
$$\hat{\eta}(t)\lesssim (1-t)^\alpha \hat{\om}\bigg(\frac{1+t}{2}\bigg)\approx \hat{\eta}\bigg(\frac{1+t}{2}\bigg).$$
Thus, $\eta\in\hD$.
Since $\frac{\hat{\eta}(t)}{(1-t)^\alpha}$ is  essentially decreasing, by Lemma B in \cite{PjRe2022afm}, $\eta\in\dD$.
The proof is complete.
\end{proof}

In \cite{PjRj2021adv}, the authors characterized the Littlewood-Paley formula  on  Bergman spaces induced   radial weights.  For the benefit of readers, we state it as follows.

\begin{otherth}\label{thA} Let $\om$ be a radial weight, $0<p<\infty$ and $k\in\N$. Then, for all $f\in H(\D)$,
$$\int_\D |f(z)|^p\om(z)dA(z)\approx \sum_{j=0}^{k-1}|f^{(j)}(0)|^p+\int_\D |f^{(k)}(z)|^p(1-|z|^2)^{kp}\om(z)dA(z) $$
if and only if $\om\in\dD$.
\end{otherth}

\begin{Lemma}\label{0405-3} Assume  $1\leq p<\infty,n\in\N, \om\in\dD$ and $Y$ is a Banach space. Let $T: A_{\om_{[np]}}^p\to Y$ be a bounded linear operator.
Then the following statements hold.
$$\|TD^{(n)}\|_{A_\om^p\to Y}\approx \|T\|_{A_{\om_{[np]}}^p\to Y},\,\,\,
\|TD^{(n)}\|_{e,A_\om^p\to Y} \approx  \|T\|_{e,A_{\om_{[np]}}^p\to Y}.$$
\end{Lemma}
\begin{proof}
Let $\eta=\om_{[np]}$. Since
\begin{align}\label{0429-1}
\|TD^{(n)}\|_{A_\om^p\to Y}=\sup_{f\not\equiv 0}\frac{\|TD^{(n)}f\|_{Y}}{\|f\|_{A_\om^p}}
=\sup_{f^{(n)}\not\equiv 0}\frac{\|TD^{(n)}f\|_{Y}}{\|f\|_{A_\om^p}},
\end{align}
by Theorem \ref{thA} we have
$$
\|TD^{(n)}\|_{A_\om^p\to Y}\geq \sup_{f\not\equiv 0 \atop{f(0)=\cdots=f^{(n-1)}(0)=0}}\frac{\|TD^{(n)}f\|_{Y}}{\|f\|_{A_\om^p}}
\approx  \sup_{f^{(n)}\not\equiv 0}\frac{\|T f^{(n)}\|_{Y}}{\|f^{(n)}\|_{A_{\eta}^p}}=\|T\|_{A_{\eta}^p\to Y},$$
and
$$\|TD^{(n)}\|_{A_\om^p\to Y}\lesssim  \sup_{f^{(n)}\not\equiv 0}\frac{\|T f^{(n)}\|_{Y}}{\|f^{(n)}\|_{A_{\eta}^p}}=\|T\|_{A_{\eta}^p\to Y}.$$
Therefore,
\begin{align}\label{0406-1}
\|TD^{(n)}\|_{A_\om^p\to Y}\approx \|T\|_{A_\eta^p\to Y}.
\end{align}

Suppose $K:A_\om^p\to Y$ is compact. Let $(If)(z)=\int_0^z f(\xi)d\xi$.   By Theorem \ref{thA},
$I^n:A_\eta^p\to A_\om^p$ is bounded. So, $KI^n:A_\eta^p\to Y$ is compact. By (\ref{0406-1}),
$$\|T\|_{e,A_\eta^p\to Y}\leq \|T-KI^n\|_{A_\eta^p\to Y}\approx \|TD^{(n)}-KI^nD^{(n)}\|_{A_\om^p\to Y}.$$
By (\ref{0429-1}) and Theorem \ref{thA},
\begin{align*}
\|TD^{(n)}-KI^nD^{(n)}\|_{A_\om^p\to Y}&=\sup_{f^{(n)}\not\equiv 0}\frac{\|(TD^{(n)}-KI^nD^{(n)})f
\|_{Y}}{\|f\|_{A_\om^p}}   \\
&\lesssim \sup_{f^{(n)}\not\equiv 0 \atop{f(0)=\cdots=f^{(n-1)}(0)=0}}\frac{\|(TD^{(n)}-KI^nD^{(n)})f\|_{Y}}{\|f\|_{A_\om^p}} \\
&=\sup_{f^{(n)}\not\equiv 0 \atop{f(0)=\cdots=f^{(n-1)}(0)=0}}\frac{\|(TD^{(n)}-K)f\|_{Y}}{\|f\|_{A_\om^p}} \\
&\leq \|TD^{(n)}-K\|_{A_\om^p\to Y}.
\end{align*}
Thus, $\|T\|_{e,A_\eta^p\to Y} \lesssim\|TD^{(n)}\|_{e,A_\om^p\to Y}$.

Conversely, suppose $K':A_\eta^p\to Y$ is compact.    By Theorem \ref{thA}, $D^{(n)}:A_\om^p\to A_\eta^p$ is bounded.
So, $K'D^{(n)}:A_\om^p\to Y$ is compact. By (\ref{0406-1}), $$\|TD^{(n)}\|_{e,A_\om^p\to Y}\leq \|TD^{(n)}-K'D^{(n)}\|_{A_\om^p\to Y}
\approx \|T-K'\|_{A_\eta^p\to Y}.$$
Therefore, $\|TD^{(n)}\|_{e,A_\om^p\to Y} \lesssim  \|T\|_{e,A_\eta^p\to Y}$.
The proof is complete.
\end{proof}

\begin{Remark} \label{0916-2} If $1\leq p,q<\infty, \om,\eta\in\dD$, $k\in\N\cup\{0\}$, for any $u\in H(\D)$ and $\vp\in S(\D)$, by Lemma \ref{0405-3},  we have
\begin{align*}
\|uD_{\vp}^{(k)}\|_{A_\om^p\to A_\eta^q}\approx \|uC_\vp\|_{A_{\om_{[kp]}}^p\to A_\eta^q},\,\,\,\,
\|uD_{\vp}^{(k)}\|_{e,A_\om^p\to A_\eta^q}\approx \|uC_\vp\|_{e,A_{\om_{[kp]}}^p\to A_\eta^q}.
\end{align*}
\end{Remark}

The norm and essential norm of $uC_\vp:A_\om^p\to A_\mu^q$ were investigated in \cite{DjLsSy2020ms}.
 To state them, we need some more notations. When $\om\in\hD$ and $0<p<\infty$, if  $\gamma>0$ is large enough, let
\begin{align*}
f_{\lambda,\gamma,\om,p}(z)=\left(\frac{1-|\lambda|^2}{1-\ol{\lambda}z}\right)^{\gamma}\frac{1}{\om(S(\lambda))^\frac{1}{p}},\, \,\lambda,z\in\D.
\end{align*}
According to \cite[Lemma 3.1]{Pj2015}, $\|f_{\lambda,\gamma,\omega,p}\|_{A_\omega^p}\approx 1$.
For brevity, we  denote $f_{\lambda,\gamma,\omega,p}$ by $f_{\lambda,\gamma}$.
When $u\in H(\D),\vp\in S(\D), 0<q<\infty, \mu\in\hD$, for any measurable  set $E\subset\D$, let
$$\nu_{u,\vp,q,\mu}(E)=\int_{\vp^{-1}(E)} |u(z)|^q\mu(z)dA(z).$$
Then, $\|uC_\vp f\|_{A_\mu^q}= \|f\|_{L_{\nu_{u,\vp,q,\mu}}^q}$.
The maximum function of  $\nu_{u,\vp,q,\mu}$ is defined by
$$M_\om(\nu_{u,\vp,q,\mu})(z)=\sup_{z\in S(a)}\frac{\nu_{u,\vp,q,\mu}(S(a))}{\om(S(a))}, \,\,\,a,z\in\D.$$

\begin{otherth}\label{thB}
Assume  $\om,\mu\in\hD$,  $u\in H(\D)$ and $\vp\in S(\D)$.
\begin{enumerate}[(i)]
  \item When $0<p\leq q<\infty$, the following estimates hold:
        $$\|uC_\vp\|_{A_\om^p\to A_\mu^q}^q  \approx  \sup_{\lambda\in\D} \int_\D |f_{\lambda,\gamma}(\vp(z))|^q  |u(z)|^q \mu(z) dA(z),$$
        and
        $$\|uC_\vp\|_{e,A_\om^p\to A_\mu^q}^q  \approx  \limsup_{|\lambda|\to 1} \int_\D |f_{\lambda,\gamma}(\vp(z))|^q  |u(z)|^q \mu(z) dA(z).$$
  \item When $0<q<p<\infty$, the following statements are equivalent:
        \begin{enumerate}
          \item $uC_\vp:{A_\om^p\to A_\mu^q}$ is bounded;
          \item $uC_\vp:{A_\om^p\to A_\mu^q}$ is compact;
          \item $\|M_\om(\nu_{u,\vp,q,\mu})\|_{L_\om^\frac{p}{p-q}}<\infty.$
        \end{enumerate}
        Moreover,
  $$\|uC_\vp\|^q _{A_\om^p\to A_\mu^q}\approx \|M_\om(\nu_{u,\vp,q,\mu})\|_{L_\om^\frac{p}{p-q}}. $$
\end{enumerate}
\end{otherth}

For a positive number $\gamma$ and $j\in\N$, let $(\gamma)_j=\gamma(\gamma+1)\cdots(\gamma+j-1)$ and $(\gamma)_0=1$.
The following lemma is a refinement of a statement  in the proof  of Theorem 3 in \cite{AsFt2019caot}.
%For the convenience of readers, we will present it as a lemma and provide the details.

\begin{Lemma}\label{0331-1}
Suppose $n\in\N\cup\{0\}$ and $M\geq 1$.  There exists a strictly increasing sequence $\{\gamma_k\}_{k=0}^n$
such that $\gamma_0$ is large enough and
\begin{align}\label{0406-3}
M+M^2\sum_{k\neq j,{0\leq k\leq n}} \gamma_k^{\frac{1}{2}-k}(\gamma_k)_j<\gamma_j^{\frac{1}{2}-j}(\gamma_j)_j,  \,\,\,\,j=0,1,\cdots,n.
\end{align}
%$$\left|\det\left(\gamma_k^{\frac{1}{2}-k}(\gamma_k)_j\right)_{k,j=0}^n\right|>1.$$
\end{Lemma}
\begin{proof}
Suppose $\gamma_k>n$ for all $k=0,1,\cdots,n$. Since $\gamma_j^\frac{1}{2}\leq \gamma_j^{\frac{1}{2}-j}(\gamma_j)_j$ and
\begin{align*}
M+M^2\sum_{k\neq j,{0\leq k\leq n}} \gamma_k^{\frac{1}{2}-k}(\gamma_k)_j
&=M+M^2\sum_{k=0}^{j-1} \gamma_k^{\frac{1}{2}-k}(\gamma_k)_j +M^2\sum_{k=j+1}^n  \gamma_k^{\frac{1}{2}-k}(\gamma_k)_j   \\
&\leq M+M^2\sum_{k=0}^{j-1} \gamma_k^{\frac{1}{2}-k}(\gamma_k)_j +M^2\sum_{k=j+1}^n  \gamma_k^{\frac{1}{2}-k}(2\gamma_k)^j   \\
&\leq M+M^2\sum_{k=0}^{j-1} \gamma_k^{\frac{1}{2}-k}(\gamma_k)_j +n2^nM^2,
\end{align*}
it is enough to choose $\{\gamma_k\}$ such that
\begin{align*}
M+M^2\sum_{k=0}^{j-1} \gamma_k^{\frac{1}{2}-k}(\gamma_k)_j +n2^n M^2  <  \gamma_j^{\frac{1}{2}},\,j=0,1,\cdots,n.
\end{align*}
When $j=0$, let $\gamma_0>(M+n2^nM^2)^2$. Suppose $\{\gamma_k\}_{k=0}^{j-1}$ is chosen. Then we can choose
$$\gamma_{j}>\left(M+M^2\sum_{k=0}^{j-1} \gamma_k^{\frac{1}{2}-k}(\gamma_k)_j +n2^nM^2 \right)^2.$$
By mathematic induction, we get the desired result. The proof is complete.
\end{proof}

\begin{Lemma}\label{lm3.2}\cite[Lemma 3.2]{AsFt2019caot}
Let $\{\lambda_j\}_{j=0}^n\subset\D$ and $\{z_j\}_{j=0}^n\subset\D$. There exists a constant $C$ depending only on $n$ such that there exists a polynomial $p$ with $\|p\|_{H^\infty}<C$  and $p^{(j)}(\lambda_j)=z_j$ for all $j=0,1,2,\cdots,n$.
\end{Lemma}

As usual, let   $\beta(\cdot,\cdot)$ be the Bergman metric,
i.e.,  for all $\xi,\eta\in\D$,
$$\beta(\xi,\eta)=\frac{1}{2}\log\frac{1+|\vp_\xi(\eta)|}{1-|\vp_\xi(\eta)|},\,\,\mbox{ where }\,\,\vp_\xi(\eta)=\frac{\xi-\eta}{1-\ol{\xi}\eta}.$$
%So, a sequence $\{a_j\}_{j=1}^\infty\subset\DD$ is a $r$-lattice for some $0<r<\infty$ means
%$\beta(a_i,a_j)\geq \frac{r}{5}$ for all $i\neq j$ and $\D=\cup_{j=1}^\infty D(a_j,5r)$.
%Here and henceforth,
%$D(z,r)=\big\{\xi\in\D; \beta(z,\xi)<r\big\}.$

\begin{Lemma}\label{lm3.5} \cite[Lemma 3.5]{AsFt2019caot}
Let $c>1, \varepsilon>0, J,M\in\N$ and $N\in\N\cup\{0\}$ be given.
Then there is a constant $C$ such that for all $\{z_j\}_{j=1}^J\subset\D$ and $\{w_m\}_{m=1}^M\subset \D$ satisfying
$$\beta(z_j,z_1)<\varepsilon, \,\,\,\,\beta(w_m,z_1)>c\varepsilon ,\,\,(1\leq j\leq J,\,\,1\leq m \leq M),$$
there exists a function $f\in H(\D)$ satisfying
$$\|f\|_{H^\infty}\leq C,\,\,\, f^{(n)}(z_j)=\delta_{0n},\,\,\,\,f^{(n)}(w_m)=0 $$
for all $0\leq n\leq N, 1\leq j\leq J, 1\leq m\leq M$.
\end{Lemma}

The following lemma can be obtained by a standard argument, see   Lemma 2.10 in \cite{Tm1996} for example, and we omit its proof here.

\begin{Lemma}\label{0406-2} Suppose $0<p,q<\infty$, $\om,\mu\in\hD$. Let $Y$ be $A_\mu^q$ or $H^\infty$. If $T:A_\om^p\to Y$ is bounded, then $T$ is compact if and only if $\|T{f_n}\|_Y\to 0$ as $n\to \infty$ whenever $\{f_n\}$ is bounded in $A_\om^p$ and uniformly converges to 0 on any compact subset of $\D$ as $n\to \infty$.
\end{Lemma}

\section{proof of Theorem  \ref{th1}}

\begin{proof}[Proof of Therem \ref{th1}]
It is obvious that
\begin{align*}
\|T_{n,\varphi,\vec{u}}\|_{A_\om^p\to A_\mu^q}\lesssim \sum_{k=0}^n\|u_k D_\vp^{(k)}\|_{A_\om^p\to A_\mu^q} \,\,\mbox{ when }\,\, 1\leq p, q<\infty
\end{align*}
and
\begin{align*}
\|T_{n,\varphi,\vec{u}}\|_{e,A_\om^p\to A_\mu^q}\lesssim \sum_{k=0}^n\|u_k D_\vp^{(k)}\|_{e,A_\om^p\to A_\mu^q} \,\,\mbox{ when }\,\, 1\leq p\leq q<\infty.
\end{align*}

Next, we only need to prove the inverse  of the above inequalities.
We first claim that
\begin{align}\label{0405-4}
\|u_0D_\vp^{(0)}\|_{A_\om^p\to A_\mu^q}\lesssim \|T_{n,\varphi,\vec{u}}\|_{A_\om^p\to A_\mu^q}\,\,\mbox{ when }\,\, 1\leq p,  q<\infty
\end{align}
and
\begin{align}\label{0912-2}\
\|u_0D_\vp^{(0)}\|_{e,A_\om^p\to A_\mu^q}\lesssim \|T_{n,\varphi,\vec{u}}\|_{e, A_\om^p\to A_\mu^q} \,\,\mbox{ when }\,\, 1\leq p\leq q<\infty.
\end{align}
Take these for granted for a moment.
%So, for any given $\om,\mu\in\dD$, $1\leq p\leq q<\infty$ and $n\in\N$, there is a constant $C$, independent of $\{u_{j}\}_{j=0}^n\subset H(\D)$ and $\vp\in S(\D)$,  (\ref{0405-4}) and (\ref{0912-2}) always hold.
Let $\widetilde{T}_{n-1,\varphi,\vec{u}}=\sum_{j=0}^{n-1} u_{j+1}D_\vp^{(j)}.$ By Lemmas \ref{0405-2} and \ref{0405-3}, we have $\om_{[p]}\in \dD$ and
$$\|\widetilde{T}_{n-1,\varphi,\vec{u}}\|_{A_{\om_{[p]}}^p\to A_\mu^q}
%\left\|\sum_{j=0}^{n-1} u_{j+1}D_\vp^{(j)}\right\|_{A_{\om_{[p]}}^p\to A_\mu^q}
\approx \Big\|\sum_{j=1}^n u_jD_\vp^{(j)}\Big\|_{A_\om^p\to A_\mu^q}=\|T_{n,\varphi,\vec{u}}-u_0D_{\vp}^{(0)}\|_{A_\om^p\to A_\mu^q}.$$
Then, (\ref{0405-4})  and Triangle Inequality deduce
$$\|\widetilde{T}_{n-1,\varphi,\vec{u}}  \|_{A_{\om_{[p]}}^p\to A_\mu^q}\lesssim \|T_{n,\varphi,\vec{u}}\|_{A_\om^p\to A_\mu^q}.$$
Since $\om_{[p]}\in \dD$, using  Lemma \ref{0405-3} and (\ref{0405-4}) again, we obtain
\begin{eqnarray*}
\|u_1 D_\vp^{(1)}\|_{A_\om^p\to A_\mu^q}
\approx\|u_1 D^{(0)}_\vp\|_{A_{\om_{[p]}}^p\to A_\mu^q}
&\lesssim& \left\|\widetilde{T}_{n-1,\varphi,\vec{u}}\right\|_{A_{\om_{[p]}}^p\to A_\mu^q}
\lesssim \|T_{n,\varphi,\vec{u}}\|_{A_\om^p\to A_\mu^q}.
\end{eqnarray*}
Then,  mathematical induction deduces
$$ \|u_j D_\vp^{(j)}\|_{A_\om^p\to A_\mu^q}
\lesssim \|T_{n,\varphi,\vec{u}}\|_{A_\om^p\to A_\mu^q},\,\,\,\, j=2,3,\cdots,n,$$
and therefore
$$ \sum_{j=0}^n \|u_j D_\vp^{(j)}\|_{A_\om^p\to A_\mu^q}
\lesssim \|T_{n,\varphi,\vec{u}}\|_{A_\om^p\to A_\mu^q}.$$
Similarly, we have
 $$ \sum_{j=0}^n \|u_j D_\vp^{(j)}\|_{e,A_\om^p\to A_\mu^q}
\lesssim \|T_{n,\varphi,\vec{u}}\|_{e, A_\om^p\to A_\mu^q}\,\,\mbox{ when }\,\, 1\leq p\leq q<\infty.$$
Here, we omit the proof of the essential norm of $T_{n,\varphi,\vec{u}}:A_\om^p\to A_\mu^q$ when $q<p$. The above proof of this theorem, Theorem \ref{thB}, Lemmas \ref{0405-2} and \ref{0405-3} ensure that $T_{n,\varphi,\vec{u}}$ and $u_kD_{\vp}^{(k)}$ are all compact when $T_{n,\varphi,\vec{u}}$ is bounded.

It  remains to prove (\ref{0405-4}) and (\ref{0912-2}). To do this, let  $T_{n,\varphi,\vec{u}}:A_\om^p\to A_\mu^q$ be bounded and
 $\{\gamma_k\}_{k=0}^n$ be those in Lemma \ref{0331-1} for $M=1$ and large enough.
For any $\lambda\in\D$ and $k=0,1,\cdots,n$, let
$$f_{\lambda,\gamma_k}(z)=\left(\frac{1-|\lambda|^2}{1-\ol{\lambda}z}\right)^{\gamma_k}\frac{1}{\om(S(\lambda))^\frac{1}{p}}, ~~~z\in H(\D).$$
%By \cite[Lemma 3.1]{Pj2015} (also see \cite[Lemma 2.4]{PjRj2014book}), we have
%$\|f_{\lambda,\gamma_i}\|_{A_\om^p}\approx 1.$
Then we have
\begin{align}\label{0326-1}
\|T_{n,\varphi,\vec{u}} f_{\lambda,\gamma_k}\|_{A_\mu^q}^q
= \int_\D \left|\sum_{j=0}^n \frac{(\gamma_k)_j(\ol{\lambda})^j u_j(z)}{(1-\ol{\lambda}\vp(z))^j}\right|^q
\left|\frac{1-|\lambda|^2}{1-\ol{\lambda}\vp(z)}\right|^ {q\gamma_k} \frac{1}{\om(S(\lambda))^{\frac{q}{p}}}\mu(z) dA(z).
\end{align}
Since $\{\gamma_k\}_{k=0}^n$ is increasing, when $k<n$, we obtain
\begin{align}\label{0330-1}
\left|\frac{1-|\lambda|^2}{1-\ol{\lambda}\vp(z)}\right|^ {q\gamma_n}
\leq 2^{q(\gamma_n-\gamma_k)}  \left|\frac{1-|\lambda|^2}{1-\ol{\lambda}\vp(z)}\right|^ {q\gamma_k}.
\end{align}
Thus, for all $k=0,1,2,\cdots,n$,
\begin{align} \label{0331-2}
\|T_{n,\varphi,\vec{u}} f_{\lambda,\gamma_k}\|_{A_\mu^q}^q
\gtrsim \int_\D \left|\sum_{j=0}^n \frac{(\gamma_k)_j(\ol{\lambda})^j u_j(z)}{(1-\ol{\lambda}\vp(z))^j}\right|^q
\left|\frac{1-|\lambda|^2}{1-\ol{\lambda}\vp(z)}\right|^ {q\gamma_n} \frac{1}{\om(S(\lambda))^{\frac{q}{p}}}\mu(z) dA(z).
\end{align}
Let  $\Delta_{j,k}=\gamma_k^{\frac{1}{2}-k}(\gamma_k)_j, j,k=0,1,\cdots,n$ and $A=(\Delta_{j,k})$.
By Ger$\check{\mbox{s}}$gorin's theorem, see \cite[Theorem 6.1.1]{HrJc1985} for example, $|\det(A)|>1$.
So, there exists a sequence $\{c_k\}_{k=0}^n$ such that
\begin{align}\label{0331-3}
A
\left(
        \begin{array}{c}
          c_0 \\
          c_1\\
          \vdots \\
          c_n\\
        \end{array}
\right)
=
\left(
\begin{array}{c}
  1 \\
  0\\
  \vdots \\
  0
\end{array}
\right).
\end{align}

% Then, we will prove (\ref{0405-4}) and (\ref{0912-2}).

{\bf Case (a). $1\leq p\leq q<\infty.$}  Using (\ref{0331-3}), it is easy to check that
{\small
\begin{align}
\|u_0D_\vp^{(0)} f_{\lambda,\gamma_n}\|_{A_\mu^q}^q
=&\int_\D \left|\left(\sum_{k=0}^n c_k \gamma_k^{\frac{1}{2}-k}(\gamma_k)_0\right)
u_0(z)f_{\lambda,\gamma_n}(\vp(z)) \right|^q \mu(z) dA(z)  \nonumber\\
=&\int_\D \left|\sum_{j=0}^n \left(
\sum_{k=0}^n c_k \gamma_k^{\frac{1}{2}-k}(\gamma_k)_j\right)
\left(\frac{(\ol{\lambda})^j u_j(z)}{(1-\ol{\lambda}\vp(z))^j}
f_{\lambda,\gamma_n}(\vp(z))\right)\right|^q \mu(z) dA(z)  \nonumber\\
\lesssim&
\sum_{k=0}^n  |c_k|^q \gamma_k^{\frac{q}{2}-kq} \int_\D \left|
\sum_{j=0}^n \left(
\frac{(\gamma_k)_j(\ol{\lambda})^j u_j(z)}{(1-\ol{\lambda}\vp(z))^j}
f_{\lambda,\gamma_n}(\vp(z))\right)\right|^q  \mu(z) dA(z). \label{0912-1}%\\
\end{align}
}
Then, (\ref{0331-2}) implies
$$\|u_0D_\vp^{(0)}f_{\lambda,\gamma_n}\|_{A_\mu^q}^q
\lesssim \sum_{k=0}^n  |c_k|^q \gamma_k^{\frac{q}{2}-kq}\|T_{n,\varphi,\vec{u}} f_{\lambda,\gamma_k}\|_{A_\mu^q}^q
\lesssim \left(\sum_{k=0}^n  |c_k|^q \gamma_k^{\frac{q}{2}-kq}\right)\|T_{n,\varphi,\vec{u}} \|_{A_\om^p\to A_\mu^q}^q.$$
By Theorem \ref{thB}, we see that $u_0D_\vp^{(0)}:A_\om^p\to A_\mu^q$ is bounded
and
$$\|u_0D_\vp^{(0)}\|_{A_\om^p\to A_\mu^q}\lesssim   \|T_{n,\varphi,\vec{u}}\|_{A_\om^p\to A_\mu^q}.$$
If  $K:A_\om^p\to A_\mu^q$ is bounded, similarly to the proof of (\ref{0912-1}), we get
%for any $\lambda\in\D$,   by (\ref{0331-3}) and Fubini's theorem,
{\small
\begin{align}
&\|(u_0D_\vp^{(0)}-K)f_{\lambda,\gamma_n}\|_{A_\mu^q}^q    \nonumber\\
=&\int_\D \left|\left(\sum_{k=0}^n c_k \gamma_k^{\frac{1}{2}-k}(\gamma_k)_0\right)
\Bigg( u_0(z)f_{\lambda,\gamma_n}(\vp(z)) -(Kf_{\lambda,\gamma_n})(z)\Bigg)\right|^q \mu(z) dA(z)  \nonumber\\
=&\int_\D \left|\sum_{j=0}^n \left(
\sum_{k=0}^n c_k \gamma_k^{\frac{1}{2}-k}(\gamma_k)_j\right)
\left(\frac{(\ol{\lambda})^j u_j(z)}{(1-\ol{\lambda}\vp(z))^j}
f_{\lambda,\gamma_n}(\vp(z))-(Kf_{\lambda,\gamma_n})(z)\right)\right|^q \mu(z) dA(z)  \nonumber\\
\lesssim&
\sum_{k=0}^n  |c_k|^q \gamma_k^{\frac{q}{2}-kq} \int_\D \left|
\sum_{j=0}^n \left(
\frac{(\gamma_k)_j(\ol{\lambda})^j u_j(z)}{(1-\ol{\lambda}\vp(z))^j}
f_{\lambda,\gamma_n}(\vp(z))-(\gamma_k)_j(Kf_{\lambda,\gamma_n})(z)\right)\right|^q  \mu(z) dA(z). \label{0429-2}%\\
\end{align}
}
Since
\begin{align*}
&\left|
\sum_{j=0}^n \left(\frac{(\gamma_k)_j(\ol{\lambda})^j u_j(z)}{(1-\ol{\lambda}\vp(z))^j}
f_{\lambda,\gamma_n}(\vp(z))-(r_k)_j(Kf_{\lambda,\gamma_n})(z) \right) \right|^q  \\
\lesssim &\left|
\sum_{j=0}^n \frac{(\gamma_k)_j(\ol{\lambda})^j u_j(z)}{(1-\ol{\lambda}\vp(z))^j}
f_{\lambda,\gamma_k}(\vp(z))-(Kf_{\lambda,\gamma_k})(z)  \right|^q \left|\frac{1-|\lambda|^2}{1-\ol{\lambda}\vp(z)}\right|^ {q(\gamma_n-\gamma_k)}\\
&+ \left|\frac{1-|\lambda|^2}{1-\ol{\lambda}\vp(z)}\right|^ {q(\gamma_n-\gamma_k)} |(Kf_{\lambda,\gamma_k})(z)|^q + \sum_{j=0}^n|(\gamma_k)_j|^q|(Kf_{\lambda,\gamma_n})(z)|^q ,
\end{align*}
by (\ref{0429-2})  we have
$$\|(u_0D_\vp^{(0)}-K)f_{\lambda,\gamma_n}\|_{A_\mu^q}^q
\lesssim \sum_{k=0}^n  |c_k|^q \gamma_k^{\frac{q}{2}-kq}
\left(\|(T_{n,\varphi,\vec{u}}-K) f_{\lambda,\gamma_k}\|_{A_\mu^q}^q
+\|K f_{\lambda,\gamma_k}\|_{A_\mu^q}^q+\|K f_{\lambda,\gamma_n}\|_{A_\mu^q}^q\right).$$
Then,  Triangle Inequality deduces
$$\|u_0D_\vp^{(0)}f_{\lambda,\gamma_n}\|_{A_\mu^q}^q
\lesssim \sum_{k=0}^n  |c_k|^q \gamma_k^{\frac{q}{2}-kq}
\left(\|(T_{n,\varphi,\vec{u}}-K) f_{\lambda,\gamma_k}\|_{A_\mu^q}^q
+\|K f_{\lambda,\gamma_k}\|_{A_\mu^q}^q+\|K f_{\lambda,\gamma_n}\|_{A_\mu^q}^q\right).$$
By Lemma \ref{0406-2}, we obtain
\begin{align*}
\limsup_{|\lambda|\to 1}\|u_0D_\vp^{(0)}f_{\lambda,\gamma_n}\|_{A_\mu^q}^q
\lesssim \sum_{k=0}^n  \limsup_{|\lambda|\to 1}
\|(T_{n,\varphi,\vec{u}}-K) f_{\lambda,\gamma_k}\|_{A_\mu^q}^q
\lesssim \|T_{n,\varphi,\vec{u}}-K\|_{A_\om^p\to A_\mu^q}^q.
\end{align*}
By Theorem \ref{thB} and the arbitrary of $K$, we have
$$\|u_0D_\vp^{(0)}\|_{e, A_\om^p\to A_\mu^q} \lesssim \|T_{n,\varphi,\vec{u}}\|_{e,A_\om^p\to A_\mu^q}.$$

{\bf Case (b): $1\leq q<p<\infty.$} Let $\{\lambda_i\}_{i=1}^\infty$ be those $\{\xi_{j,l}^k\}$ in \cite[Thoerem 2]{PjRjSk2021jga}(also see \cite[Theorem 3.2]{ZxXlFhLj2014amsc})
%and\textcolor[rgb]{1.00,0.00,0.00}{ $r_k(t) =$sign$(\sin(2kπt))$
and  $r_k(t)$  be  Rademacher functions. We have the following statements:
\begin{enumerate}
  \item[(bi).] For any $a=\{a_i\}_{i=1}^\infty\in l^p$, $\|g_{a,\gamma_k,t}\|_{A_\om^p}\lesssim \|\{a_i\}_{i=1}^\infty\|_{l^p}$,  where
           $$g_{a,\gamma_k,t}(z)=\sum_{i=1}^\infty a_i r_i(t)f_{\lambda_i,\gamma_k}(z),\,\,k=0,1,\cdots,n.  $$
  \item[(bii).] For any $g\in A_\om^{\frac{p}{2}}$, there exists $\{b_i\}_{i=1}^{\infty}\in l^\frac{p}{2}$ such that
  \begin{align}\label{0912-3}
  g(z)=\sum_{i=1}^\infty b_i\left(\frac{1-|\lambda_i|^2}{1-\ol{\lambda_i}z}\right)^{2\gamma_n}\frac{1}{\om(S(\lambda_i))^{\frac{2}{p}}},
\,\,\,\,\|g\|_{A_\om^\frac{p}{2}}\approx \|\{b_i\}\|_{l^\frac{p}{2}}.
\end{align}
\end{enumerate}

By Fubini's theorem, Khinchin's inequality,  (\ref{0326-1}) and (\ref{0330-1}), we have
 \begin{align*}
&\int_0^1 \|T_{n,\varphi,\vec{u}} g_{a,\gamma_k,t}\|_{A_\mu^q}^qdt\\
=&\int_\D\int_0^1 \left|\sum_{i=1}^\infty a_i r_i(t) (T_{n,\varphi,\vec{u}} f_{\lambda_i,\gamma_k})(z) \right|^q   dt  \mu(z) dA(z)\\
\approx& \int_\D\left( \sum_{i=1}^\infty |a_i|^2|(T_{n,\varphi,\vec{u}} f_{\lambda_i,\gamma_k})(z)|^2\right)^\frac{q}{2}  \mu(z) dA(z)\\
=&
\int_\D\left( \sum_{i=1}^\infty \left| a_i
\left(\sum_{j=0}^n \frac{(\gamma_k)_j(\ol{\lambda_i})^j u_j(z)}{(1-\ol{\lambda_i}\vp(z))^j}  \right)
\left(\frac{1-|\lambda_i|^2}{1-\ol{\lambda_i}\vp(z)}\right)^ {\gamma_k} \frac{1}{\om(S(\lambda_i))^{\frac{1}{p}}}
\right|^2 \right)^\frac{q}{2}  \mu(z) dA(z)
\end{align*}
\begin{align*}
\gtrsim&
\int_\D\left( \sum_{i=1}^\infty \left|
a_i
\left(\sum_{j=0}^n \frac{(\gamma_k)_j(\ol{\lambda_i})^j u_j(z)}{(1-\ol{\lambda_i}\vp(z))^j}  \right)
\left(\frac{1-|\lambda_i|^2}{1-\ol{\lambda_i}\vp(z)}\right)^ {\gamma_n} \frac{1}{\om(S(\lambda_i))^{\frac{1}{p}}}
\right|^2
\right)^\frac{q}{2}  \mu(z) dA(z)
\\
=& \int_\D \left( \sum_{i=1}^\infty  |A_{k,i}(z)|^2\right)^\frac{q}{2}\mu(z) dA(z).
\end{align*}
Here $$
A_{k,i}(z)=a_i
\left(\sum_{j=0}^n \frac{(\gamma_k)_j(\ol{\lambda_i})^j u_j(z)}{(1-\ol{\lambda_i}\vp(z))^j}  \right)
\left(\frac{1-|\lambda_i|^2}{1-\ol{\lambda_i}\vp(z)}\right)^ {\gamma_n} \frac{1}{\om(S(\lambda_i))^{\frac{1}{p}}}.
$$
Thus, by the statement (bi),
\begin{align*}
 \int_\D \left( \sum_{i=1}^\infty  |A_{k,i}(z)|^2\right)^\frac{q}{2}\mu(z) dA(z)
 &\lesssim \int_0^1 \|T_{n,\varphi,\vec{u}} g_{a,\gamma_k,t}\|_{A_\mu^q}^qdt\\
 &\lesssim  \|T_{n,\varphi,\vec{u}}\|_{A_\om^p\to A_\mu^q}^q\int_0^1  \| g_{a,\gamma_k,t}\|_{A_\om^p}^qdt\\
 &
 \lesssim  \|T_{n,\varphi,\vec{u}}\|_{A_\om^p\to A_\mu^q}^q  \|\{a_i\}\|_{l^p}^q.
\end{align*}
Recalling  that $\{c_k\}_{k=0}^n$ was decided by (\ref{0331-3}), we have
 \begin{align}
&\int_\D\left( \sum_{i=1}^\infty \left|a_i \right|^2
\left|(u_0D_\vp^{(0)} f_{\lambda_i,\gamma_n})(z)\right|^2
\right)^\frac{q}{2}  \mu(z) dA(z)\nonumber\\
=& \int_\D \left( \sum_{i=1}^\infty  \left|\sum_{k=0}^n c_k \gamma_k^{\frac{1}{2}-k}A_{k,i}(z)\right|^2\right)^\frac{q}{2}\mu(z) dA(z)  \nonumber\\
\lesssim& \int_\D \left( \sum_{i=1}^\infty \sum_{k=0}^n  |A_{k,i}(z)|^2\right)^\frac{q}{2}\mu(z) dA(z)  \nonumber\\
\lesssim&  \sum_{k=0}^n\int_\D \left( \sum_{i=1}^\infty  |A_{k,i}(z)|^2\right)^\frac{q}{2}\mu(z) dA(z) \nonumber
\\
\lesssim&  \|T_{n,\varphi,\vec{u}}\|_{A_\om^p\to A_\mu^q}^q \|\{a_i\}\|_{l^p}^q.  \label{0330-2}%\|a\|_{T_2^p(\{\lambda_k\},\om)}^q.
\end{align}
For any $g\in A_\om^{\frac{p}{2}}$, by the statement (bii), there exists $\{b_i\}_{i=1}^\infty\in l^\frac{p}{2}$ such that (\ref{0912-3}) holds.
Let $a_i=b_i^\frac{1}{2}$.
So, $\|\{b_i\}\|_{l^\frac{p}{2}}=\|\{a_i\}\|_{l^p}^2$.
By (\ref{0330-2}), we get
\begin{align*}
\|u_0^2 D_\vp^{(0)} g\|_{A_\mu^\frac{q}{2}}^\frac{q}{2}
\leq& \int_\D\left( \sum_{i=1}^\infty \left|a_i\right|^2
\left|(u_0 D_\vp^{(0)} f_{\lambda_i,\gamma_n})(z)\right|^2
\right)^\frac{q}{2}  \mu(z) dA(z)\\
\lesssim& \|T_{n,\varphi,\vec{u}}\|_{A_\om^p\to A_\mu^q}^q \|\{a_i\}\|_{l^p}^q
=\|T_{n,\varphi,\vec{u}}\|_{A_\om^p\to A_\mu^q}^q \|\{b_i\}\|_{l^\frac{p}{2}}^\frac{q}{2}\\
\approx &\|T_{n,\varphi,\vec{u}}\|_{A_\om^p\to A_\mu^q}^q \|g\|_{A_\om^\frac{p}{2}}^\frac{q}{2}.
\end{align*}
That is to say, $u_0^2 D_\vp^{(0)}:A_\om^\frac{p}{2}\to A_\mu^\frac{q}{2}$ is bounded.
By  Theorem \ref{thB}, $u_0D_\vp^{(0)}:A_\om^p\to A_\mu^q$ is bounded and
$$\|u_0 D_\vp^{(0)}\|_{A_\om^p\to A_\mu^q}\lesssim \|T_{n,\varphi,\vec{u}}\|_{A_\om^p\to A_\mu^q}.$$
The proof is complete.
\end{proof}

\section{Proof of Theorem \ref{th3}}

\begin{proof}[Proof of Theorem \ref{th3}]
For brief, let $\upsilon(a)=\om(S(a))^\frac{1}{p}$. By Chapter 1 in \cite{PjRj2014book}, there exist $\alpha,\beta>0$ such that
$\frac{\upsilon(t)}{(1-t)^\alpha}$ and $\frac{\upsilon(t)}{(1-t)^\beta}$ are essentially increasing and essentially decreasing on $[0,1)$, respectively.
By (4.8) in \cite{Zhu1}, there is a constant $M_0>4$, whenever $\beta(z,w)<\frac{1}{2}$,
\begin{align*}
\frac{2}{\sqrt{M_0}}<\frac{\upsilon(z)}{\upsilon(w)}<\frac{\sqrt{M_0}}{2}.
\end{align*}
By Lemma \ref{0331-1}, we can choose $\{\gamma_j\}_{j=0}^n$ large enough  such that (\ref{0406-3}) holds for $M=M_0$.
By Proposition 4.5 in \cite{Zhu1}, there exist $0<\varepsilon<\frac{1}{4}<R<1$ such that whenever $|z|>R$, $\xi,\eta\in D(z,\varepsilon)$ and $0\leq k,j\leq n$,
\begin{align}\label{0516-1}
\frac{1}{\sqrt{M_0}}<|\xi|^n, |\eta|^n<1,
\end{align}
\begin{align*}
\frac{1}{2}<
\frac{(1-|\xi|^2)^{\gamma_j}(1-|\eta|^2)^k}{|1-\ol{\xi}\eta|^{\gamma_j+k}}<2,
\end{align*}
and then,
\begin{align}\label{0218-1}
\frac{1}{\sqrt{M_0}}<
\frac{\upsilon(\eta)}{\upsilon(\xi)}
\frac{(1-|\xi|^2)^{\gamma_j}(1-|\eta|^2)^k}{|1-\ol{\xi}\eta|^{\gamma_j+k}}<\sqrt{M_0}.
\end{align}

{\bf Case (a):} $|\lambda_J|>R$ and $\sup\limits_{0\leq j\leq n}\beta(\lambda_j,\lambda_J)<\varepsilon$.
Let $f_{\lambda,\gamma}(z)=\left(\frac{1-|\lambda|^2}{1-\ol{\lambda}z}\right)^\gamma\frac{1}{\upsilon(\lambda)}$ and
$$f_{\Lambda,J}(z)=\sum_{k=0}^n b_k \gamma_k^{\frac{1}{2}-k}f_{\lambda_k,\gamma_k}(z).$$
Then, the  equations (\ref{IC}) can be written as  $Ab=\delta_J$, in which
$$a_{jk}=\gamma_k^{\frac{1}{2}-k}(\gamma_k)_j\ol{\lambda_k}^j
\frac{\upsilon(\lambda_j)}{\upsilon(\lambda_k)}
\frac{(1-|\lambda_k|^2)^{\gamma_k}(1-|\lambda_j|^2)^j}{(1-\ol{\lambda_k}\lambda_j)^{\gamma_k+j}}, $$
and
$$
b=(b_0,b_1,\cdots,b_n)^T,\,\,\,\,\,\,
\delta_J=(\delta_{0J},\,\, \delta_{1J}, \cdots, \,\, \delta_{nJ})^T,  \,\,\,\,\, A=(a_{jk}).$$
By (\ref{0406-3}),  (\ref{0516-1}) and (\ref{0218-1}), for any $0\leq j\leq n$,
\begin{align*}
1+\sum_{k\neq j,0\leq k\leq n}|a_{jk}|
< 1+\sqrt{M_0}\sum_{k\neq j,0\leq k\leq n} \gamma_k^{\frac{1}{2}-k} (\gamma_k)_j
< \frac{1}{M_0}\gamma_j^{\frac{1}{2}-j}(\gamma_j)_j
<|a_{jj}|.
\end{align*}
By Ger$\check{\mbox{s}}$gorin's theorem, $|\det(A)|>1$.
Meanwhile, since all   elements in $A$ are bounded independent of $\{\lambda_j\}_{j=0}^n$, the elements of adjoint matrix $A^*$ of $A$ are also bounded. So, there exists a constant $M_1$ independent of  $\{\lambda_j\}_{j=0}^n$ such that
$b=A^{-1}\delta_J$ and $\sum_{k=0}^n |b_k|<M_1$.
So,  $\|f_{\Lambda,J}\|_{A_\om^p}\lesssim M_1 $. For a fixed $J\in\{0,1,\cdots,n\}$ and any given $\delta\in (0,1)$ and $0\leq j\leq n$, set
$$E_{j,\delta}=\left\{\lambda_j:\Lambda=\{\lambda_k\}_{k=0}^n\subset \D, |\lambda_J|>\delta,
 \sup_{0\leq i\leq n}\beta(\lambda_i,\lambda_J)<\varepsilon\right\}.$$
As $\delta$ approaches $1$, $E_{j,\delta}$ approaches the boundary of $\D$.
Therefore, the functions $\{f_{\lambda_j,\gamma_j}\}_{\lambda_j\in E_{j,\delta}}$  converge to $0$   uniformly on any compact  subset of $\D$.
By the arbitrary of $0\leq j\leq n$ and $\sum_{k=0}^n |b_k|<M_1$,
the functions $\{f_{\Lambda,J}\}$ converge to $0$   uniformly on any compact  subset of $\D$ as $|\lambda_J|$ approaches $1$.

{\bf Case (b):} $|\lambda_J|>R$ and $\sup\limits_{0\leq k\leq n}\beta(\lambda_k,\lambda_J)\geq \varepsilon$.
Let $\varepsilon^\p=\frac{\varepsilon}{n+1}$.
By Pigeonhole Principle, there exists $L\in\{0,1,\cdots,n\}$ such that
$$\Big\{\lambda_j:L\varepsilon^\p\leq \beta(\lambda_j,\lambda_J)<(L+1)\varepsilon^\p, j=0,1,2,\cdots,n\Big\}$$
is empty.
Set $\Lambda_1=\{z_k\}_{k=0}^n$, where
$$z_k=\left\{
\begin{array}{cc}
  \lambda_k, & \mbox{ if }\,\, \beta(\lambda_k,\lambda_J)<L\varepsilon^\p, \\
  \lambda_J, & \mbox{ if }\,\, \beta(\lambda_k,\lambda_J)\geq (L+1)\varepsilon^\p.
\end{array}
\right.
$$
By the proof above,  we have a function $f_{\Lambda_1,J}$ such $\|f_{\Lambda_1,J}\|_{A_\om^p}\lesssim M_1$ and
\begin{align}\label{0517-1}
f_{\Lambda_1,J}^{(j)}(z_j)=\frac{\delta_{jJ}}{(1-|z_j|^2)^{j}\upsilon(z_j)},\,\,0\leq j\leq n.
\end{align}
Let $w_1,w_2,\cdots,w_{n^\p}$ be the elements of $\{\lambda_k\}_{k=0}^n \backslash \{z_k\}_{k=0}^n$.
By Lemma \ref{lm3.5}, there exist a constant $M_2$,  independent of $\{\lambda_j\}_{j=0}^n$,  $J$,  and $h\in H(\D)$ such that
for all $0\leq k,j\leq n$, $0\leq i\leq n^\p$, we have
\begin{align}\label{0218-3}
\|h\|_{H^\infty}<M_2,\,\,\,\,  h^{(k)}(w_i)=0, h^{(k)}(z_j)=\delta_{0k}.
\end{align}
Letting $f_{\Lambda,J}=f_{\Lambda_1,J}h$, for all $j=0,1,\cdots,n$, we have
$$f_{\Lambda,J}^{(j)}(z)=\sum_{k=0}^j C_j^k f_{\Lambda_1,J}^{(k)}(z)h^{(j-k)}(z).$$
When $\lambda_j=z_j$, by (\ref{0517-1}) and (\ref{0218-3}), we have
$$f_{\Lambda,J}^{(j)}(\lambda_j)=f_{\Lambda,J}^{(j)}(z_j)=f_{\Lambda_1,J}^{(j)}(z_j)
=\frac{\delta_{jJ}}{(1-|z_j|^2)^{j}\upsilon(z_j)}
=\frac{\delta_{jJ}}{(1-|\lambda_j|^2)^{j}\upsilon(\lambda_j)};$$
otherwise,  $\lambda_j$ could not be $\lambda_J$ and there exists $w_i$ such that $\lambda_j=w_i$.  By (\ref{0218-3}),
$$f_{\Lambda,J}^{(j)}(\lambda_j)=f_{\Lambda,J}^{(j)}(w_i)=0
=\frac{\delta_{jJ}}{(1-|\lambda_j|^2)^{j}\upsilon(\lambda_j)}.$$
So, $f_{\Lambda,J}=f_{\Lambda_1,J}h$ is the desired and $\|f_{\Lambda,J}\|_{A_\om^p}\lesssim  M_1 M_2$.
Moreover, by the proof of case (a) and $\|h\|_{H^\infty}<M_2$, the functions $\{f_{\Lambda,J}\}$ converge to $0$ uniformly on any compact subset of $\D$ when $|\lambda_J|$ approaches $1$.

{\bf Case (c):}  $|\lambda_J|\leq R$. By Lemma \ref{lm3.2}, there is a constant $M_3$,
for all $\{\lambda_j\}_{j=0}^n$ and $0\leq J\leq n$, there is a function $p\in H^\infty$ such that $\|p\|_{H^\infty}<M_3$  and $p^{(j)}(\lambda_j)=\frac{1}{2}\delta_{jJ}$.
Then $f(z)=\frac{2p(z)}{(1-|\lambda_J|^2)^J\upsilon(\lambda_J)}$ is the desired.

 By the above proof, we see that the functions $\{f_{\Lambda,J}\}$ converge to $0$ uniformly on compact subsets of $\D$ as $|\lambda_J|\to 1$.
The proof is complete.
\end{proof}

\section{Proof of Theorem \ref{th2}}

\begin{proof}[Proof of Theorem \ref{th2}]  First we consider the   norm of  $\T_{n,\vec{\varphi},\vec{u}}:A_\om^p\to H^\infty$.  By the assumption and Lemma 3 in \cite{ZxDj2019mia}, we see that  for any $f\in A_\om^p$ and $k=0,1,\cdots,n+1$,
\begin{align}\label{0407-3}
|f^{(k)}(z)|\lesssim \frac{\|f\|_{A_\om^p}}{(1-|z|^2)^k\om(S(z))^\frac{1}{p}}, ~~z\in\D.
\end{align}
 After a   calculation, by (\ref{0407-3})  we get
\begin{align*}
\|u_kD_{\vp_k}^{(k)}\|_{A_\om^p\to H^\infty} \lesssim \sup_{z\in\D} \frac{|u_k(z)|}{(1-|\vp_k(z)|^2)^k\om(S(\vp_k(z)))^\frac{1}{p}}.
\end{align*}
Therefore,
$$\|\T_{n,\vec{\varphi},\vec{u}}\|_{A_\om^p\to H^\infty}\lesssim \sum_{k=0}^n \sup_{z\in\D} \frac{|u_k(z)|}{(1-|\vp_k(z)|^2)^k\om(S(\vp_k(z)))^\frac{1}{p}}.$$

Conversely, suppose $\T_{n,\vec{\varphi},\vec{u}}:A_\om^p\to H^\infty$ is bounded.
By Theorem \ref{th3}, there exists $M'$ such that, for any $\lambda\in\D$ and $J\in\{0,1,\cdots,n\}$,  there is a function $f_{\Lambda,J}\in A_\om^p$ satisfying
$\|f_{\Lambda,J}\|_{A_\om^p}\leq M'$ and
\begin{align*}
f_{\Lambda,J}^{(j)}(\vp_j(\lambda))=\frac{\delta_{jJ}}{(1-|\vp_j(\lambda)|^2)^j\om(S(\vp_j(\lambda)))^\frac{1}{p}},\,\,j=0,1,\cdots,n.
\end{align*}
Here, $\Lambda=\{\vp_k(\lambda)\}_{k=0}^n$. Then we have
\begin{align*}
\Big|\frac{u_J(\lambda)}{(1-|\vp_J(\lambda)|^2)^J\om(S(\vp_J(\lambda)))^{\frac{1}{p}}} \Big|
=|(\T_{n,\vec{\varphi},\vec{u}} )f_{\Lambda,J}(\lambda)|
\leq \|\T_{n,\vec{\varphi},\vec{u}} f_{\Lambda,J}\|_{H^\infty}
\lesssim \|\T_{n,\vec{\varphi},\vec{u}}\|_{A_\om^p\to H^\infty}.
\end{align*}
Therefore,
$$\sum_{k=0}^n \sup_{z\in\D} \frac{|u_k(z)|}{(1-|\vp_k(z)|^2)^k\om(S(\vp_k(z)))^\frac{1}{p}}  \lesssim \|\T_{n,\vec{\varphi},\vec{u}}\|_{A_\om^p\to H^\infty},$$
as desired.

Next, we estimate the essential norm of  $\T_{n,\vec{\varphi},\vec{u}}:A_\om^p\to H^\infty$.   Suppose $\T_{n,\vec{\varphi},\vec{u}}:A_\om^p\to H^\infty$ is bounded. By the above proof, we see that
$$\sup_{z\in\D}|u_k(z)|<\infty, \,\,k=0,1,2,\cdots,n.$$
%
%$$\sum_{k=0}^\infty \sup_{z\in\D} \frac{|u_k(z)|}{(1-|\vp_k(z)|^2)^k\om(S(\vp_k(z)))^\frac{1}{p}}<\infty$$
For any given $r\in[0,1)$, let $(K_r f)(z)=f(rz)$. By Lemma \ref{0406-2}, it is compact on $A_\om^p$.
So, $u_kD_{\vp_k}^{(k)}K_r:A_\om^p\to H^\infty$ is also compact.
Let $0<\delta<1$.
For $f\in A_\om^p$, we have
\begin{align*}
\|u_kD_{\vp_k}^{(k)} f-u_kD_{\vp_k}^{(k)}K_r f\|_{H^\infty}
&\leq \left(\sup_{|\vp_k(z)|\leq \delta}  +\sup_{\delta<|\vp_k(z)|<1 }\right)|u_k(z)D_{\vp_k}^{(k)}f(z)-r^ku_k(z)D_{r\vp_k}^{(k)}f(z)|\\
&\leq  \sup_{|\vp_k(z)|\leq \delta} |u_k(z)D_{\vp_k}^{(k)}f(z)-r^ku_k(z)D_{r\vp_k}^{(k)}f(z)|\\
&  ~~~+
\sup_{\delta<|\vp_k(z)|<1 } |u_k(z)D_{\vp_k}^{(k)}f(z)-r^ku_k(z)D_{r\vp_k}^{(k)}f(z)|\\
&:= I+II.
\end{align*}
By (\ref{0407-3}), there exists a constant $M$ independent of $f,r,\delta,u_k,\vp_k$ such that
\begin{align*}
I\leq& \sup_{|\vp_k(z)|\leq \delta} (1-r^k)|u_k(z)f^{(k)}(r\vp_k(z))|
+\sup_{|\vp_k(z)|\leq \delta}|u_k(z)|\left|\int_{r\vp_k(z)}^{\vp_k(z)} f^{(k+1)}(\xi)d\xi\right|  \\
\leq&
(1-r^k+1-r)\sup_{|\vp_k(z)|\leq \delta}  \frac{M|u_k(z)|\|f\|_{A_\om^p}}{(1-|\vp_k(z)|^2)^{k+1}\om(S(\vp_k(z)))^\frac{1}{p}}
\end{align*}
and
\begin{align*}
II\leq& \sup_{\delta<|\vp_k(z)|<1 }|u_k(z)f^{(k)}(\vp_k(z))|   +  \sup_{\delta<|\vp_k(z)|<1 }|u_k(z)f^{(k)}(r\vp_k(z))|    \\
\leq& \sup_{\delta<|\vp_k(z)|<1} \frac{M|u_k(z)|\|f\|_{A_\om^p}}{(1-|\vp_k(z)|^2)^k\om(S(\vp_k(z)))^\frac{1}{p}}.
\end{align*}
So, for any given $\varepsilon>0$, we can choose $r\in(0,1)$ such that
 \begin{align*}
\|u_kD_{\vp_k}^{(k)} f-u_kD_{\vp_k}^{(k)}K_r f\|_{H^\infty}
\leq \varepsilon \|f\|_{A_\om^p} +\sup_{\delta<|\vp_k(z)|<1} \frac{M|u_k(z)|\|f\|_{A_\om^p}}{(1-|\vp_k(z)|^2)^k\om(S(\vp_k(z)))^\frac{1}{p}}.
\end{align*}
%Therefore,
%$$\|u_kD_{\vp_k}^{(k)}\|_{e,A_\om^p\to H^\infty}
%\leq \|u_kD_{\vp_k}^{(k)} -u_kD_{\vp_k}^{(k)}K_r \|_{A_\om^p\to H^\infty}
%\lesssim \varepsilon+ \sup_{\delta<|\vp_k(z)|<1} \frac{|u_k(z)|}{(1-|\vp_k(z)|^2)^k\om(S(\vp_k(z)))^\frac{1}{p}}.
%$$
Letting $\varepsilon\to 0$ and $\delta\to 1$, we have
$$\|u_kD_{\vp_k}^{(k)}\|_{e,A_\om^p\to H^\infty} \lesssim \limsup_{|\vp_k(z)|\to 1} \frac{|u_k(z)|}{(1-|\vp_k(z)|^2)^k\om(S(\vp_k(z)))^\frac{1}{p}}.$$
Therefore,
$$\|\T_{n,\vec{\varphi},\vec{u}}\|_{e, A_\om^p\to H^\infty}\lesssim \sum_{k=0}^\infty
\limsup_{|\vp_k(z)|\to 1} \frac{|u_k(z)|}{(1-|\vp_k(z)|^2)^k\om(S(\vp_k(z)))^\frac{1}{p}}.$$

Finally, we prove that
\begin{align*}
\sum_{j=0}^n\limsup_{|\vp_j(\lambda)|\to 1}
\frac{|u_j(\lambda)|}{(1-|\vp_j(\lambda)|^2)^j\om(S(\vp_j(\lambda)))^{\frac{1}{p}}} \lesssim \|\T_{n,\vec{\varphi},\vec{u}}\|_{e,A_\om^p\to H^\infty}.
\end{align*}
Suppose $K:A_\om^p\to H^\infty$ is compact.
By Theorem \ref{th3}, there exists $M'$ such that, for any $\lambda\in\D$ and $J\in\{0,1,\cdots,n\}$,  there is a function $f_{\Lambda,J}\in A_\om^p$ satisfying
$\|f_{\Lambda,J}\|_{A_\om^p}\leq M'$ and
\begin{align*}
f_{\Lambda,J}^{(j)}(\vp_j(\lambda))=\frac{\delta_{jJ}}{(1-|\vp_j(\lambda)|^2)^j\om(S(\vp_j(\lambda)))^\frac{1}{p}},\,\,j=0,1,\cdots,n.
\end{align*}
Here, $\Lambda=\{\vp_k(\lambda)\}_{k=0}^n$.
Then we have
\begin{align*}
\Big|\frac{u_J(\lambda)}{(1-|\vp_J(\lambda)|^2)^J\om(S(\vp_J(\lambda)))^{\frac{1}{p}}} -(Kf_{\Lambda,J})(\lambda)\Big|
&=|(\T_{n,\vec{\varphi},\vec{u}}-K)f_{\Lambda,J}(\lambda)|\\
&\leq \|\T_{n,\vec{\varphi},\vec{u}} f_{\Lambda,J}-Kf_{\Lambda,J}\|_{H^\infty}\\
&\lesssim \|\T_{n,\vec{\varphi},\vec{u}}-K\|_{A_\om^p\to H^\infty}.
\end{align*}
By  Lemma \ref{0406-2} and Theorem \ref{th3}, $\|Kf_{\Lambda,J}\|_{H^\infty}\to 0$ as $|\vp_J(\lambda)|\to 1$.
Thus,
\begin{align*}
\limsup_{|\vp_J(\lambda)|\to 1}\frac{|u_J(\lambda)|}{(1-|\vp_J(\lambda)|^2)^J\om(S(\vp_J(\lambda)))^{\frac{1}{p}}} \lesssim \|\T_{n,\vec{\varphi},\vec{u}}-K\|_{A_\om^p\to H^\infty}.
\end{align*}
Since $K$ and $J$ are arbitrary, we have
\begin{align*}
\sum_{j=0}^n\limsup_{|\vp_j(\lambda)|\to 1}
\frac{|u_j(\lambda)|}{(1-|\vp_j(\lambda)|^2)^j\om(S(\vp_j(\lambda)))^{\frac{1}{p}}} \lesssim \|\T_{n,\vec{\varphi},\vec{u}}\|_{e,A_\om^p\to H^\infty}.
\end{align*}
The proof is complete.
\end{proof}

\end{document}